\documentclass{amsart}

\usepackage{amsmath, amssymb, amsthm, setspace, tikz-cd}

\theoremstyle{plain}
\newtheorem{theorem}{Theorem}[section]
\newtheorem{lemma}[theorem]{Lemma}

\newtheorem{corollary}[theorem]{Corollary}
\newtheorem{fact}[theorem]{Fact}

\theoremstyle{definition}
\newtheorem{definition}[theorem]{Definition}
\newtheorem{exampletemp}[theorem]{Example}
\newtheorem{remark}[theorem]{Remark}

\newtheorem*{acknowledgements}{Acknowledgements}

\newenvironment{example}{\begin{exampletemp}}{\hfill\qed\end{exampletemp}}
\begin{document}

\title[Finite Topological Graph Algebras]{Finiteness Properties of Certain Topological Graph Algebras}
\author{Christopher P. Schafhauser}
\address{Department of Mathematics \\ University of Nebraska -- Lincoln \\ 203 Avery Hall \\ Lincoln, NE 68588-0130}
\email{cschafhauser2@math.unl.edu}
\date{\today}

\begin{abstract}
  Let $E = (E^0, E^1, r, s)$ be a topological graph with no sinks such that $E^0$ and $E^1$ are compact.  We show that when $C^*(E)$ is finite, there is a natural isomorphism $C^*(E) \cong C(E^\infty) \rtimes \mathbb{Z}$, where $E^\infty$ is the infinite path space of $E$ and the action is given by the backwards shift on $E^\infty$.  Combining this with a result of Pimsner, we show the properties of being AF-embeddable, quasidiagonal, stably finite, and finite are equivalent for $C^*(E)$ and can be characterized by a natural ``combinatorial'' condition on $E$.
\end{abstract}

\maketitle

\section{Introduction}\label{sec:Introduction}

A topological graph $E = (E^0, E^1, r, s)$ is a directed graph such that the vertex set $E^0$ and the edge set $E^1$ are both locally compact Hausdorff spaces and the range and source maps $r, s: E^1 \rightarrow E^0$ satisfy appropriate continuity conditions (see Section \ref{sec:TopologicalGraphs}).  A graph can be viewed as a generalization of a dynamical system where the graph $E$ is thought of as a partially defined, multi-valued, continuous map $E^0 \rightarrow E^0$ given by $v \mapsto r(s^{-1}(v)) \subseteq E^0$ for every $v \in s(E^0) \subseteq E^0$.

In \cite{KatsuraTGA1, KatsuraTGA2, KatsuraTGA3, KatsuraTGA4}, Katsura constructs a $C^*$-algebra $C^*(E)$ from a topological graph $E$ and gives a very detailed analysis of these $C^*$-algebras.  This construction generalizes both discrete graph algebras and homeomorphism $C^*$-algebras.  As with discrete graphs, the structure of the graph algebra $C^*(E)$ is closely related to the structure of the underlying graph $E$.  For instance, the ideal structure, $K$-theory, simplicity, and pure infiniteness of $C^*(E)$ can all be described in terms of the graph $E$.  There are also generalizations of the Cuntz-Krieger and Gauge Invariant Uniqueness Theorems for topological graphs.

Although the class of $C^*$-algebras which are defined by discrete graphs is fairly limited, it seems that many interesting $C^*$-algebras (especially those in the classification program) appear as topological graph algebras.  For instance, every crossed product $C_0(X) \rtimes \mathbb{Z}$, every UCT Kirchberg algebra, and every AF-algebra appear as topological graph algebras.  In fact, there does not seem to be any known nuclear UCT $C^*$-algebras which do not arise as a topological graph algebras, although it seems unlikely that every nuclear UCT $C^*$-algebra will have this form.  For instance, in view of Theorem \ref{thm:TopGraphsAFEmbedding} below, it seems plausible that a finite $C^*$-algebra that is not stably finite will not be a topological graph algebra.

We are interested in the finiteness properties of $C^*(E)$. In particular, when is $C^*(E)$ AF-embeddable, quasidiagonal, stable finite, or finite?  A $C^*$-algebra is called \emph{quasidiagonal} if there is a net of completely positive contractive maps $\varphi_n: A \rightarrow F_n$ where each $F_n$ is a finite dimensional $C^*$-algebra and the $\varphi_n$ are asymptotically isometric and multiplicative; i.e.\ for every $a, b \in A$, we have
\[ \| \varphi_n(ab) - \varphi_n(a) \varphi_n(b) \| \rightarrow 0 \qquad \text{and} \qquad \|\varphi(a_n)\| \rightarrow \|a\|. \]
It follows from Arveson's Extension Theorem that every AF-algebra is quasidiagonal.  Subalgebras of quasidiagonal $C^*$-algebras are certainly quasidiagonal and hence every AF-embeddable $C^*$-algebra is quasidiagonal.  Moreover, it is well known that quasidiagonal $C^*$-algebras are stably finite (\cite[Proposition 7.1.15]{BrownOzawa}).

The converses are known to be false.  In particular, by a result of Choi (\cite[Theorem 7]{Choi}), $C^*(\mathbb{F}_2)$ is quasidiagonal (in fact the $\varphi_n$ may be chosen to be multiplicative) but $C^*(\mathbb{F}_2)$ is not AF-embeddable since it is not exact (\cite{Wassermann}).  Similarly, if $G$ is a non-amenable group, then $C^*_r(G)$ is not quasidiagonal by a result of Rosenberg (\cite{Rosenberg}), but it is well known that $C^*_r(G)$ is stably finite since it has a faithful trace.  Blackadar and Kirchberg conjectured in \cite{BlackadarKirchberg} that the converses may be true with some amenability assumptions.  In particular, they conjecture AF-embeddability, quasidiagonality, and stable finiteness are equivalent for nuclear $C^*$-algebras.  The main result of \cite{SchafhauserGA} verifies this conjecture for graph $C^*$-algebras.  Also, the following result from Pimsner verifies the conjecture for certain crossed products.

\begin{theorem}[Theorem 9 in \cite{PimsnerAFE}]\label{thm:Pimsner}
Suppose $X$ is a compact metric space and $f$ is a homeomorphism of $X$.  Then the following are equivalent:
\begin{enumerate}
  \item $C(X) \rtimes_f \mathbb{Z}$ is AF-embeddable;
  \item $C(X) \rtimes_f \mathbb{Z}$ is quasidiagonal;
  \item $C(X) \rtimes_f \mathbb{Z}$ is stably finite;
  \item $C(X) \rtimes_f \mathbb{Z}$ is finite;
  \item Every point in $X$ is pseudoperiodic for $f$:  given $x_1 \in X$, there are points $x_2, \ldots, x_n \in X$ such that $d(f(x_i), x_{i+1}) < \varepsilon$ for all $i = 1, \ldots, n$, where the subscripts are taken modulo $n$.
\end{enumerate}
\end{theorem}

There is a similar theorem of N. Brown for crossed products of AF-algebras by $\mathbb{Z}$ and there are partial results for many other crossed products (see \cite{BrownAFE}).  See \cite[Chapters 7 and 8]{BrownOzawa} for a survey of quasidiagonality and AF-embeddability.  Our goal in this paper is to prove a version of Theorem \ref{thm:Pimsner} for the $C^*$-algebra $C^*(E)$ generated by a compact topological graph $E$ with no sinks (Theorem \ref{thm:TopGraphsAFEmbedding}).  We show that if $C^*(E)$ is finite, then $E$ has no sources and every vertex in $E$ emits exactly one edge.  In this case, $C^*(E) \cong C(E^\infty) \rtimes \mathbb{Z}$, where $E^\infty$ is the space of all infinite paths in $E$ and the actions is given by the backward shift on $E^\infty$ (Theorems \ref{thm:Finiteness} and \ref{thm:InfinitePathCrossedProduct}).  Combining this with Theorem \ref{thm:Pimsner} will give our result.

Pimsner's proof that $(4) \Rightarrow (5)$ in Theorem \ref{thm:Pimsner} involves decomposing $X$ into its orbits (viewed as discrete spaces) and representing the elements of $C(X)$ as weighted bilateral shifts on the $\ell^2$ spaces of these orbits.  Our techniques are similar but are more involved since a topological graph can be significantly more complicated than a dynamical system.  Given an infinite path $\alpha$ in a topological graph $E$, we define a directed tree $\Gamma_\alpha$ which is thought of as the orbit of $\alpha$.  In Section \ref{sec:Representations} we  modify a construction of Katsura to represent $C^*(E)$ on $\ell^2(\Gamma_\alpha)$ by viewing the elements of $C_0(E^0)$ as diagonal operators and the elements of $C_c(E^1)$ as weighted shifts on $\ell^2(\Gamma_\alpha)$.  Weighted shifts on directed trees were defined and extensively studied in \cite{WeightedShifts}.  In particular, it was shown that the Fredholm theory of weighted shifts on directed a tree is closely related to the combinatorial structure of the tree (see Corollary \ref{cor:SurjectiveShift}).  This allows us to relate the finiteness of $C^*(E)$ to the ``combinatorial'' structure of the infinite path space $E^\infty$ and hence also to the graph $E$.

The paper is organized as follows.  In Sections \ref{sec:CuntzPimsner} and \ref{sec:TopologicalGraphs} we recall the necessary definitions and results about Cuntz-Pimsner algebras and topological graph algebras.  Section \ref{sec:WeightedShifts} contains the necessary results from \cite{WeightedShifts} on weighted shifts on directed trees.  In Section \ref{sec:Representations}, we build representations of topological graph algebras on weighted trees, and finally Section \ref{sec:Finiteness} contains our main results.

We will adopt the following conventions throughout the paper.  For graph algebras, we follow the conventions in \cite{Raeburn}. In particular, the partial isometries go in the same direction as the edges.  We assume $C^*$-algebras are separable, spaces and second countable, and graphs are countable.  If $X$ is a set, let $(\delta_x)_{x \in X}$ denote the standard orthonormal basis of $\ell^2(X)$.  That is, for $x \in X$, $\delta_x \in \ell^2(X)$ is given by $\delta_x(x) = 1$ and $\delta_x(y) = 0$ when $x \neq y$.

\begin{acknowledgements}
The author would like to thank his advisors Allan Donsig and David Pitts for their encouragement and for pointing out several errors in earlier versions of this paper.
\end{acknowledgements}

\section{Cuntz-Pimsner Algebras}\label{sec:CuntzPimsner}

Cuntz-Pimsner algebras were introduced by Pimsner in \cite{PimsnerCorr} and were expanded on by Katsura in \cite{KatsuraCorr}.  This class of algebras was further studied in \cite[Chapter 8]{Raeburn} and \cite[Section 4.6]{BrownOzawa}.  For the readers convenience, we recall the necessary definitions.

A (\emph{right}) \emph{Hilbert module} over a $C^*$-algebra $A$ is a right $A$-module $X$ together with an \emph{inner product} $\langle \cdot, \cdot \rangle :X \times X \rightarrow A$ such that for every $\xi, \eta, \zeta \in X$, $a \in A$, and $\alpha \in \mathbb{C}$,
\begin{enumerate}
  \item $\langle\xi, \xi\rangle \geq 0$ with equality if and only if $\xi = 0$,
  \item $\langle\xi, \alpha \eta + \zeta\rangle = \alpha \langle\xi, \eta\rangle + \langle\xi, \zeta\rangle$,
  \item $\langle\xi, \eta a\rangle = \langle\xi, \eta\rangle a$, and
  \item $\langle\xi, \eta\rangle^* = \langle\eta, \xi\rangle$,
\end{enumerate}
and $X$ is complete with respect to the norm $\|\xi\|^2 = \|\langle\xi, \xi\rangle\|$.

An operator $T: X \rightarrow X$ is called \emph{adjointable} if there is an operator $T^*:X \rightarrow X$ such that
\[ \langle\xi, \eta\rangle = \langle\xi, T^*\eta\rangle \]
for every $\xi, \eta \in X$.  If $T$ is adjointable, then $T^*$ is unique, and $T$ and $T^*$ are both $A$-linear and bounded.  The collection $\mathbb{B}(X)$ of all adjointable operators on $X$ is a $C^*$-algebra.  A \emph{$C^*$-correspondence} over $A$ is a Hilbert $A$-module $X$ together with a *-homomorphism $\lambda: A \rightarrow \mathbb{B}(X)$.  We often write $\lambda(a)(\xi) = a\xi$ for $a \in A$ and $\xi \in X$.

A \emph{Toeplitz representation} of $X$ on a $C^*$-algebra $B$ is a pair $(\pi, \tau)$ where $\pi:A \rightarrow B$ is a *-homomorphism, $\tau:X \rightarrow B$ is a linear map such that for every $a \in A$ and $\xi, \eta \in X$,
\[ \tau(a \xi) = \pi(a) \tau(\xi) \qquad \text{and} \qquad \tau(\xi)^*\tau(\eta) = \pi(\langle\xi, \eta\rangle). \]
Note that for every $a \in A$ and $\xi \in X$, we have $\tau(\xi a) = \tau(\xi) \pi(a)$ since
\[ \|\tau(\xi a) - \tau(\xi) \pi(a)\|^2 = \|\pi(\langle\xi a, \xi a\rangle - \langle\xi a, \xi\rangle a - a^*\langle\xi, \xi a\rangle + a^* \langle\xi, \xi\rangle a) \|^2=0 \]
Moreover, the computation
\[ \|\tau(\xi)\|^2 = \|\tau(\xi)^* \tau(\xi)\| = \|\pi(\langle\xi, \xi\rangle)\| \leq \|\langle\xi, \xi\rangle\| = \|\xi\|^2, \]
implies $\|\tau\| \leq 1$ and if $\pi$ is injective, then $\tau$ is isometric.  Let $\mathcal{T}(X)$ be the universal $C^*$-algebra generated by a Toeplitz representation of $X$.  More precisely, $\mathcal{T}(X)$ is a $C^*$-algebra together with a Toeplitz representation $(\widetilde{\pi}, \widetilde{\tau})$ of $X$ on $\mathcal{T}(X)$ such that for any other Toeplitz representation $(\pi, \tau)$ on a $C^*$-algebra $B$, there is a unique morphism $\psi: \mathcal{T}(X) \rightarrow B$ such that $\psi \circ \tilde{\pi} = \pi$ and $\psi \circ \tilde{\tau} = \tau$.
\[ \begin{tikzcd} A \arrow{r}{\widetilde{\pi}} \arrow[bend right]{ddr}[swap]{\pi} & \mathcal{T}(X) \arrow[dotted]{dd}{\exists\hskip .5pt ! \,\psi} & X \arrow{l}[swap]{\widetilde{\tau}} \arrow[bend left]{ddl}{\tau} \\ & & \\ & B & \end{tikzcd} \]
The $C^*$-algebra $\mathcal{T}(X)$ is called the \emph{Toeplitz-Pimsner algebra} associated to $X$.

For $\xi, \eta \in X$, define $\theta_{\xi, \eta} = \xi \langle\eta, \cdot\rangle: X \rightarrow X$ and let $\mathbb{K}(X)$ be the norm closed span of the $\theta_{\xi, \eta}$ in $\mathbb{B}(X)$.  Then $\mathbb{K}(X)$ is an ideal in $\mathbb{B}(X)$ and given a Toeplitz representation $(\pi, \tau)$ of $X$ on $B$, there is a unique *-homomorphism $\varphi:\mathbb{K}(X) \rightarrow B$ such that $\varphi(\theta_{\xi, \eta}) = \tau(\xi) \tau(\eta)^*$ for every $\xi, \eta \in X$.  Moreover, for each $a \in A$, $k \in \mathbb{K}(X)$, and $\xi \in X$,
\[ \pi(a)\varphi(k) = \varphi(\lambda(a)k), \quad \varphi(k)\pi(a) = \varphi(k\lambda(a)), \quad \text{and} \quad \varphi(k)\tau(\xi) = \tau(k\xi). \]
If $\pi$ is injective, then $\varphi$ is also injective.

If $X$ is a $C^*$-correspondence over $A$, define an ideal $J_X \subseteq A$ by
\[ J_X = \lambda^{-1}(\mathbb{K}(X)) \cap \{ a \in A : a b = 0 \text{ for every } b \in \ker(\lambda) \}. \]
A Toeplitz representation $(\pi, \tau)$ is \emph{covariant} if for every $a \in J_X$, $\varphi(\lambda(a)) = \pi(a)$.  Let $\mathcal{O}(X)$ denote the universal $C^*$-algebra generated by a covariant Toeplitz representation of $X$ as with the Toeplitz-Pimsner algebra $\mathcal{T}(X)$.  The $C^*$-algebra $\mathcal{O}(X)$ is called the \emph{Cuntz-Pimsner algebra} associated to $X$.

It can be shown that the Cuntz-Pimnser algebra $\mathcal{O}(X)$ always exists and the canonical maps $\pi: A \rightarrow \mathcal{O}(X)$ and $\tau: X \rightarrow \mathcal{O}(X)$ are injective and hence isometric.  Moreover, $\mathcal{O}(X)$ is unique up to a canonical isomorphism and is generated as a $C^*$-algebra by $\pi(A)$ and $\tau(A)$.  The same is true for the Toeplitz-Pimsner algebra $\mathcal{T}(X)$.  There are also concrete descriptions of $\mathcal{O}(X)$ and $\mathcal{T}(X)$ (see \cite[Section 4.6]{BrownOzawa} for example), but for our purposes, the abstract definition is more helpful.

\section{Topological Graphs}\label{sec:TopologicalGraphs}

In this section, we recall the necessary material from topological graphs.  All the material is taken from \cite{KatsuraTGA1}.

A \emph{topological graph} $E = (E^0, E^1, r, s)$ consists of locally compact second countable spaces $E^i$ and continuous maps $r, s : E^1 \rightarrow E^0$ such that $s$ is a local homeomorphism; i.e.\ for every $e \in E^1$, there is an open neighborhood $U \subseteq E^1$ of $e$ such that $s|_U$ is a homeomorphism of $U$ onto $s(U) \subseteq E^0$ and $s(U)$ is an open neighborhood of $s(e)$.  It can be shown that a local homeomorphism is always an open map.

Suppose $E$ is a topological graph.  Given $a \in C_0(E^0)$ and $\xi, \eta \in C_c(E^1)$, define
\begin{align*}
   (a \xi)(e) &= a(r(e))\xi(e)  && \text{for every $e \in E^1$,} \\
   (\xi a)(e) &= \xi(e)a(s(e))  && \text{for every $e \in E^1$, and} \\
   \langle\xi, \eta\rangle(v) &= \sum_{e \in s^{-1}(v)} \overline{\xi(e)} \eta(e)  && \text{for every $v \in E^0$.}
\end{align*}
Define $X(E)$ to be the $C^*$-correspondence over $C_0(E^0)$ obtained from completing $C_c(E^1)$ and let $C^*(E) = \mathcal{O}(X(E))$.

\begin{remark}
Note that the sum in the definition of $\langle \xi, \eta \rangle$ above is a finite sum.  To see this, suppose $v \in E^0$.  If $e \in E^1$ is an accumulation point of $s^{-1}(v) \subseteq E^0$, then there is a sequence $(e_n) \subseteq s^{-1}(v)$ such that $e_n \rightarrow e$.  with $e_n \neq e$ for every $n$.  Choose an open set $U \subseteq E^1$ such that $e \in U$.  Then there is an $n \geq 1$ such that $e_n \in U$.  Now, $s(e_n) = s(e)$ and hence $s|_U$ is not injective.  This contradicts the fact that $s$ is a local homeomorphism.  Therefore $s^{-1}(v)$ has no accumulation points.  Thus if $K \subseteq E^1$ is compact, $K \cap s^{-1}(v)$ if finite for every $v \in E^0$.  In particular, if $\xi, \eta \in C_c(E^1)$, then $\operatorname{supp}(\xi) \cap s^{-1}(v)$ and $\operatorname{supp}(\eta) \cap s^{-1}(v)$ are finite sets.  The claim follows.
\end{remark}

\begin{remark}
Given $\xi, \eta \in C_c(E^1)$, it is not immediately obvious that $\langle \xi, \eta \rangle$ is a continuous function on $E^0$.  This is the content of Lemma 1.5 in \cite{KatsuraTGA1}.  The lemma relies heavily on the fact that $s$ is a local homeomorphism.
\end{remark}

\begin{example}
When $E^0$ and $E^1$ are both discrete, then $E$ is a graph in the sense of \cite{Raeburn} and $C^*(E)$ is the usual graph $C^*$-algebra.
\end{example}

\begin{example}
Suppose $X$ is a locally compact Hausdorff space and $f$ is a homeomorphism of $X$.  Define a topological graph $E$ by $E^0 = E^1 = X$, $r = \operatorname{id}$ and $s = f$.  Then $C^*(E) \cong C(X) \rtimes_f \mathbb{Z}$.

In general one can think of $C^*(E)$ as a crossed product by a partially defined, multi-valued, continuous map $E^0 \rightarrow E^0$ given by $v \mapsto r(s^{-1}(v)) \subseteq E^0$ for each $v \in s(E^1) \subseteq E^0$.
\end{example}

Given $v \in E^0$, we say $v$ is \emph{regular} if there is a neighborhood $V$ of $v$ such that $r^{-1}(V)$ is compact and $r(r^{-1}(V)) = V$.  Let $E^0_{reg}$ denote the collection of regular vertices in $E^0$.  Note that $E^0_{reg}$ is the largest open subset $U \subseteq E^0$ such that $r$ restricts to a proper surjection from $r^{-1}(U)$ onto $U$.  Moreover, for each $v \in E^0_{reg}$, the set $r^{-1}(v) \subseteq E^1$ is compact and non-empty.  The elements of $E^0_{sng} := E^0 \setminus E^0_{reg}$ are called the \emph{singular vertices} of $E$.  A vertex $v \in E^0$ is called a \emph{sink} if $s^{-1}(v) = \emptyset$ and a \emph{source} if $v \in E^0 \setminus \overline{r(E^1)}$.

\begin{theorem}[Theorem 1.24 in \cite{KatsuraTGA1}]
For a topological graph $E$, $J_{X(E)} = C_0(E^0_{reg})$.  In particular, a Toepltiz representation $(\pi, \tau)$ of $X(E)$ is covariant if and only if $\pi(a) = \varphi(\lambda(a))$ for each $a \in C_0(E^0_{reg})$.
\end{theorem}

A \emph{path} if $E$ is a list of edges $\alpha = e_n e_{n-1} \cdots e_1$ in $E^1$ such that if $1 \leq k < n$, $r(e_k) = s(e_{k+1})$.  If also $r(e_n) = s(e_1)$, then $\alpha$ is called a \emph{loop}.  Set $s(\alpha) = s(e_1)$ and $r(\alpha) = r(e_n)$.  The integer $n$ is called the \emph{length} of $\alpha$ and is written $n = |\alpha|$.  Let $E^n$ denote the collection of all paths in $E$ with length $n$. Equip $E^n$ with the product topology and let $E^* = \coprod_{n=0}^\infty E^n$.  Then $r:E^* \rightarrow E^0$ is continuous and $s:E^* \rightarrow E^0$ is a local homeomorphism.

The following is known as the Gauge Invariance Uniqueness Theorem.

\begin{theorem}[Theorem 4.5 in \cite{KatsuraTGA1}]\label{thm:GIUT}
Suppose $(\pi, \tau)$ is a covariant representation of $E$ on a $C^*$-algebra $A$ and suppose $A$ has a gauge action $\gamma: \mathbb{T} \rightarrow \operatorname{Aut} A$ such that $\gamma_z(\pi(a)) = \pi(a)$ for every $a \in C_0(E^0)$, $z \in \mathbb{T}$ and $\gamma_z(\tau(\xi)) = z \tau(\xi)$ for every $\xi \in X(E)$, $z \in \mathbb{T}$.  Then the induced map $C^*(E) \rightarrow A$ is injective if and only if $\pi$ is injective.
\end{theorem}

\section{Weighted Shifts on Directed Trees}\label{sec:WeightedShifts}

All of the results in this section are taken from \cite{WeightedShifts}.

\begin{definition}
A \emph{directed tree} is a (countable, discrete) graph $E$ such that
\begin{enumerate}
  \item $E$ has no loops,
  \item $r$ is injective, and
  \item $E$ is connected; i.e.\ given $v, w \in E^0$, there are vertices $v_0, v_1, \ldots, v_n \in E^0$ such that $v_0 = v$, $v_n = w$, and there are edges $e_1, \ldots, e_n \in E^1$ such that either $s(e_i) = v_i$ and $r(e_i) = v_{i-1}$ or $r(e_i) = v_i$ and $s(e_i) = v_{i-1}$.
\end{enumerate}
\end{definition}

The connectedness condition in the above definition says any two vertices should be connected by an \emph{undirected} path.

\begin{fact}
If $E$ is a directed tree, then $E$ has at most one source.
\end{fact}

\begin{definition}
Suppose $E$ is a directed tree and $\lambda = (\lambda_v)_{v \in E^0} \subseteq \mathbb{C}$. Given $f \in \ell^2(E^0)$, define $\Lambda_E f: E^0 \rightarrow \mathbb{C}$ by
\[ (\Lambda_E f)(v) = \begin{cases} \lambda_v f(s(r^{-1}(v))) & v \in r(E^1), \\ 0 & v \notin r(E^1). \end{cases} \]
Now define a (possibly unbounded) operator $S_{\lambda}$ on $\ell^2(E^0)$ by
\[ \mathcal{D}(S_{\lambda}) = \{ f \in \ell^2(E^0) : \Lambda_E f \in \ell^2(E^0) \} \]
and $S_{\lambda}f = \Lambda_E f$ for $f \in \mathcal{D}(S_{\lambda})$.
$S_{\lambda}$ is called a \emph{weighted shift} on $E$.
\end{definition}

We are only interested in the case $S_{\lambda}$ is bounded.  Fortunately, there is a simple characterization of the weights $\lambda$ which define bounded operators.  All the results stated in this section are also true when $S_{\lambda}$ is an unbounded densely defined operator.

\begin{theorem}
If $E$ is a directed tree and $\lambda = (\lambda_v)_{v \in E^0} \subseteq \mathbb{C}$, then the following are equivalent:
\begin{enumerate}
  \item $\mathcal{D}(S_{\lambda}) = \ell^2(E^0)$,
  \item $S_{\lambda} \in \mathbb{B}(\ell^2(E^0))$,
  \item $\sup_{v \in E^0} \sum_{w \in r(s^{-1}(v))} |\lambda_w|^2 < \infty$.
\end{enumerate}
In this case,
\[ \| S_{\lambda} \|^2 = \sup_{v \in E^0} \sum_{w \in r(s^{-1}(v))} |\lambda_w|^2. \]
\end{theorem}

\begin{example}
Given the graph
\[ \begin{tikzcd} \cdots \arrow{r} & -2 \arrow{r} & -1 \arrow{r} & 0 \arrow{r} & 1 \arrow{r} & 2 \arrow{r} & \cdots \end{tikzcd} \]
then $S_{\lambda}$ is the weighted bilateral shift on $\ell^2(\mathbb{Z})$ given by $S_{\lambda}\delta_n = \lambda_{n+1} \delta_{n+1}$ for every $n \in \mathbb{Z}$.
\end{example}

\begin{example}
Suppose $E$ is the graph below:

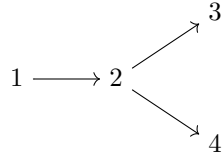
\begin{center}
\begin{tikzcd}[row sep = small]
            &                        & 3 \\
1 \arrow{r} & 2 \arrow{ru}\arrow{rd} &   \\
            &                        & 4
\end{tikzcd}
\end{center}

\noindent For any $\lambda \in \mathbb{C}^4$, we have
\[ S_{\lambda} = \begin{pmatrix} 0 & 0 & 0 & 0 \\ \lambda_2 & 0 & 0 & 0 \\ 0 & \lambda_3 & 0 & 0 \\ 0 & \lambda_4 & 0 & 0 \end{pmatrix} \]
\end{example}

\begin{theorem}
Let $S_{\lambda}$ be a bounded weighted shift on a directed tree $E$.  Then $S_{\lambda}$ has closed range if and only if
\[ \inf \{ \| S_{\lambda} \delta_v \| : v \in E^0 \text{ and } S_{\lambda} \delta_v \neq 0 \} > 0. \]
\end{theorem}

\begin{theorem}
Let $S_{\lambda}$ be a bounded weighted shift on a directed tree $E$.  Then $S_{\lambda}$ is injective if and only if $S_{\lambda} \delta_v \neq 0$ for every $v \in E^0$.
\end{theorem}

\begin{corollary}\label{cor:BoundedBelow}
Let $S_{\lambda}$ be a bounded weighted shift on a directed tree $E$.  Then $S_{\lambda}$ is bounded below if and only if
\[ \varepsilon := \inf \{ \| S_{\lambda} \delta_v \| : v \in E^0 \} > 0. \]
In this case, $\|S_{\lambda} f\| \geq \varepsilon \|f\|$ for every $f \in \ell^2(E^0)$.
\end{corollary}

It turns out that the Fredholm theory for the weighted shifts $S_{\lambda}$ is closely related to the combinatorial structure of the underlying graph.  The following theorem gives a simple formula for the index of a Fredholm weighted shift.  This result will be very important for understanding finite topological graph algebras since it gives an easy way of constructing left-invertible operators which are not right-invertible.

\begin{theorem}
Let $S_{\lambda}$ be a bounded weighted shift on a directed tree $E$.  Define
$X = \{ v \in E^0 : S_{\lambda} \delta_v \neq 0 \}$.  Then
\begin{enumerate}
  \item $\dim (\ker S_{\lambda}) = \left| E^0 \setminus X \right|$, and
  \item $\displaystyle \dim (\ker S_{\lambda}^*) = \sum_{v \in X} (|s^{-1}(v)| - 1) + \sum_{v \in E^0 \setminus X} |s^{-1}(v)| + |E^0 \setminus r(E^0)|$.
\end{enumerate}
\end{theorem}

\begin{corollary}\label{cor:SurjectiveShift}
Suppose $S_{\lambda}$ is a bounded weighted shift on $E$ that is bounded below.  Then $S_{\lambda}$ is surjective if and only if $r: E^1 \rightarrow E^0$ is surjective and $s:E^1 \rightarrow E^0$ is injective.
\end{corollary}

\section{Representations on Directed Trees}\label{sec:Representations}

Our goal in this section is to build a family of representations of a topological graph on directed trees as defined in the previous section.  Roughly, given an infinite path $\alpha$ in $E$, we will define the ``orbit'' of $\alpha$ to be a discrete directed tree $\Gamma_\alpha$ inside the infinite path space $E^\infty$. The elements of $C_0(E^0)$ and $X(E)$ will be represented, respectively, as diagonal operators and weighted shifts on the Hilbert space $\ell^2(\Gamma_\alpha)$. The construction is a slight modification of the representations built by Katsura (see Remark \ref{remark:KatsuraConstruction} below).

Let $E$ be a topological graph and set
\[ E^\infty := \{ \alpha: \mathbb{N} \rightarrow E^1 : s(\alpha(n)) = r(\alpha(n+1)) \text{ for every } n \in \mathbb{N} \}, \]
and given $\alpha \in E^\infty$, define $r(\alpha) = r(\alpha(1))$.  An element $\alpha$ of $E^\infty$ is thought of as an infinite path as shown below
\[ \begin{tikzcd} \cdots \arrow{r}{\alpha(3)} & \bullet \arrow{r}{\alpha(2)} & \bullet \arrow{r}{\alpha(1)} & r(\alpha) \end{tikzcd} \]
in $E$.  Given a path $\beta \in E^*$ with $s(\beta) = r(\alpha)$, we define $\beta\alpha \in E^\infty$ by
\[ (\beta\alpha)(n) = \begin{cases} \beta(n) & n \leq |\beta| \\ \alpha(n - |\beta|) & n > |\beta|. \end{cases} \]
We think of $\beta\alpha$ as being the path given by ``shifting forward'' by $\beta$.  There is also a notion of ``backward shift'' given by $\sigma: E^\infty \rightarrow E^\infty$, $\sigma(\alpha)(n) = \alpha(n+1)$.

For each $\alpha \in E^\infty$ and $n \in \mathbb{Z}$, define
\[ \Gamma_\alpha^n = \{ \beta \sigma^k(\alpha) : k \in \mathbb{N},\, \beta \in E^{n+k},\, s(\beta) = r(\alpha(k+1)) \} \]
Set $\Gamma_\alpha = \coprod_{n \in \mathbb{Z}} \Gamma_\alpha^n$.  Note that for each $\alpha \in E^\infty$ and $n \in \mathbb{Z}$, $\sigma(\Gamma_\alpha^n) \subseteq \Gamma_\alpha^{n-1}$.  The elements of $\Gamma_\alpha$ are all the infinite paths obtained by taking shifts of $\alpha$.  An element of $\Gamma_\alpha^n$ is thought of as an infinite path of ``length'' $|\alpha| + n$.

The set $\Gamma_\alpha$ can be thought of as a directed tree $F$ where $F^0 = F^1 = \Gamma_\alpha$, $r = \mathrm{id}$, and $s = \sigma$.

\begin{remark}\label{remark:KatsuraConstruction}
In \cite[Section 12]{KatsuraTGA3} Katsura used a similar construction to study the primitive ideals in $C^*(E)$.  In particular, given $\alpha \in E^\infty$, define
\[ \operatorname{Orb}(\alpha) = \{ \beta \sigma^k(\alpha) : k \in \mathbb{N},\, \beta \in E^*,\, s(\beta) = r(\alpha(k+1)) \} \subseteq E^\infty. \]
As with $\Gamma_\alpha$, we may think of $\operatorname{Orb}(\alpha)$ as a (discrete) directed graph.  When $\alpha$ is periodic in the sense that there is a loop $\beta \in E^*$ and a path $\gamma \in E^*$ with $\alpha = \gamma \beta^\infty$, then $\operatorname{Orb}(\alpha)$ will have a loop and hence is distinct from $\Gamma_\alpha$.  When $\alpha$ is aperiodic, there is no difference between $\operatorname{Orb}(\alpha)$ and $\Gamma_\alpha$.
\end{remark}

\begin{example}
If $e \in E^1$ is a loop; i.e.\ $s(e) = r(e)$, then we may form an infinite path $\alpha = e^\infty = e e e \cdots$.  Then in $E^\infty$, $\sigma(\alpha) = \alpha$, and hence $\sigma(\alpha) = \alpha$ in $\operatorname{Orb}(\alpha)$.  However in $\Gamma_\alpha$, $\sigma(\alpha)$ and $\alpha$ are viewed as different elements of $\Gamma_\alpha$.  In particular, $\alpha \in \Gamma_\alpha^0$ and $\sigma(\alpha) \in \Gamma_\alpha^{-1}$.  Similarly $e\alpha$ and $\alpha$ are considered to be different elements in $\Gamma_\alpha$ since $e\alpha \in \Gamma_\alpha^1$.
\end{example}

\begin{example}
Let $E$ be the graph shown below.
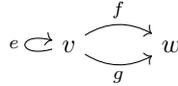
\[ \begin{tikzcd} v \arrow[loop left]{}{e} \arrow[bend left]{r}{f} \arrow[bend right]{r}[swap]{g} & w \end{tikzcd} \]
Then $E^\infty = \{e^\infty, f e^\infty, g e^\infty\}$.  The graph $\Gamma_{e^\infty}$ is given by
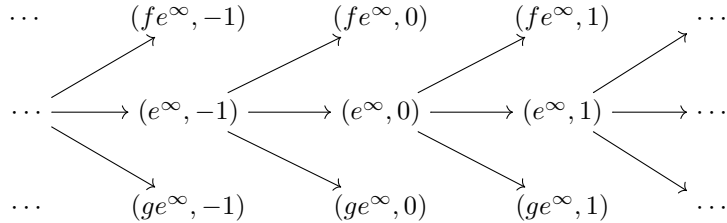
\begin{center}
\begin{tikzcd}
\cdots & (f e^\infty, -1) & (f e^\infty, 0) & (f e^\infty, 1) & \cdots \\
\cdots \arrow{r} \arrow{dr} \arrow{ur} & (e^\infty, -1) \arrow{r} \arrow{dr} \arrow{ur} & (e^\infty, 0) \arrow{r} \arrow{dr} \arrow{ur} & (e^\infty, 1) \arrow{r} \arrow{dr} \arrow{ur} & \cdots \\
\cdots & (g e^\infty, -1) & (g e^\infty, 0) & (g e^\infty, 1) & \cdots
\end{tikzcd}
\end{center}
In particular, $\Gamma_{e^\infty}^n = \{(e^\infty, n), (fe^\infty, n), (ge^\infty, n)\}$.  On the other hand, the graph $\operatorname{Orb}(e^\infty)$ is given by
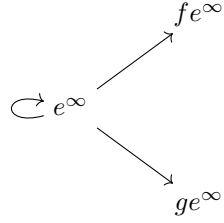
\[ \begin{tikzcd}[column sep = small] & & f e^\infty \\ e^\infty \arrow[loop left]{} \arrow{rru} \arrow{rrd}& & \\ & & g e^\infty \end{tikzcd} \]
\end{example}

For $\alpha \in E^\infty$, define $H_\alpha = \ell^2(\Gamma_\alpha)$ and $H_\alpha^n = \ell^2(\Gamma_\alpha^n)$.  Then $H_\alpha = \bigoplus_{n \in \mathbb{Z}} H_\alpha^n$ and hence $H_\alpha$ may be viewed as a $\mathbb{Z}$-graded Hilbert space.  The $\mathbb{Z}$-grading on $H_\alpha$, induces a natural gauge action on $\mathbb{B}(H_\alpha)$ which we now describe.  For each $z \in \mathbb{T}$, define $U_{\alpha, z} \in \mathbb{B}(H_\alpha)$ by $U_{\alpha, z} h = z^n h$ for each $h \in H_\alpha^n$.  Now set $\gamma_{\alpha, z} = \operatorname{Ad}(U_{\alpha, z}) \in \operatorname{Aut} \mathbb{B}(H_\alpha)$.  It is easy to verify $\gamma_\alpha: \mathbb{T} \rightarrow \operatorname{Aut} \mathbb{B}(H_\alpha)$ is a point-norm continuous group homomorphism.

We will build a gauge invariant representation $\psi_\alpha : C^*(E) \rightarrow \mathbb{B}(H_\alpha)$ for each $\alpha \in E^\infty$.  Fix $\alpha \in E^\infty$.  Define maps $\pi_\alpha: C_0(E^0) \rightarrow \mathbb{B}(H_\alpha)$ and $\tau_\alpha:C_c(E^1) \rightarrow \mathbb{B}(H_\alpha)$ by
\begin{equation*}
\pi_\alpha(a)\delta_\beta = a(r(\beta)) \delta_\beta, \qquad \text{and} \qquad \tau_\alpha(\xi)\delta_\beta = \sum_{e \in s^{-1}(r(\beta))} \xi(e) \delta_{e\beta}.
\end{equation*}
for every $a \in C_0(E^0)$, $\xi \in X(E)$, and $\beta \in \Gamma_\alpha$.

Although our orbit space $\Gamma_\alpha$ is slightly different than Katusra's orbit space $\operatorname{Orb}(\alpha)$ (see Remark \ref{remark:KatsuraConstruction}), the definition of $(\pi_\alpha, \tau_\alpha)$ and the proof that this is a covariant Toeplitz representation is identical to Katusra's proof in \cite[Section 12]{KatsuraTGA3}.  For the readers convenience we outline a proof.

A routine (but slightly tedious) argument shows that $(\pi_\alpha, \tau_\alpha)$ is a Toeplitz representation.  Let $\varphi_\alpha: \mathbb{K}(X(E)) \rightarrow \mathbb{B}(H_\alpha)$ be the *-homomorphism given by $\varphi_\alpha(\theta_{\xi, \eta}) = \tau_\alpha(\xi)\tau_\alpha(\eta)^*$.  To show $(\pi_\alpha, \tau_\alpha)$ is covariant, we must show $\pi_\alpha(a) = \varphi_\alpha(\lambda(a))$ whenever $a \in C_c(E^0_{reg})$.

Suppose $a \in C_c(E^0_{reg})$.  Since $\operatorname{supp}(a)$ is a compact subset of $E^0_{reg}$, the set $K = r^{-1}(\operatorname{supp}(a)) \subseteq E^1$ is compact.  Choose $e_1, \ldots, e_m \in K$ and relatively compact open sets $U_1, \ldots, U_m \subseteq K$ such that $e_i \in U_i$, $s|_{U_i}$ is a homeomorphism, and $K \subseteq U_1 \cup U_2 \cup \cdots \cup U_m$.  Choose continuous functions $\zeta_i \in C_c(E^1)$ such that $0 \leq \zeta_i \leq 1$, $\operatorname{supp} \zeta_i \subseteq U_i$, and $\zeta_1(e) + \cdots \zeta_m(e) = 1$ for each $e \in K$.

For $i = 1, \ldots, m$, set $\xi_i = (a \circ r) \zeta_i^{1/2}$ and $\eta_i = \zeta_i^{1/2}$.  Note that for $e \in E^1$ and $\zeta \in C_c(E^1)$,
\begin{align*}
   \sum_{i=1}^m \theta_{\xi_i, \eta_i}(\zeta)(e)
&= \sum_{i=1}^m \xi_i(e) \sum_{f \in s^{-1}(s(e))} \overline{\eta_i(f)} \zeta(f)
 = \sum_{i=1}^m \xi_i(e) \overline{\eta_i(e)} \zeta(e) \\
&= \sum_{i=1}^m a(r(e)) \zeta_i(e) \zeta(e)
 = a(r(e))\zeta(e)
 = \lambda(a)(\zeta)(e).
\end{align*}
Hence $\lambda(a) = \sum_{i=1}^m \theta_{\xi_i, \eta_i}$.  Now,
\begin{align*}
   \varphi_\alpha(\lambda(a)) \delta_\beta
&= \sum_{i=1}^m \tau_\alpha(\xi_i)\tau_\alpha(\eta_i)^* \delta_\beta
 = \sum_{i=1}^m \tau_\alpha(\xi_i) \overline{\eta_i(\beta(1))} \delta_{\sigma(\beta)} \\
&= \sum_{i=1}^m \sum_{e \in s^{-1}(r(\sigma(\beta)))} \xi_i(e) \overline{\eta_i(\beta(1))} \delta_{e\sigma(\beta)} \\
&= \sum_{i=1}^m \xi_i(\beta(1)) \overline{\eta_i(\beta(1))} \delta_{\beta} \\
&= a(r(\beta(1))) \delta_\beta
 = a(r(\beta)) \delta_\beta
 = \pi_\alpha(a) \delta_\beta.
\end{align*}
Hence the Toeplitz representation $(\pi_\alpha, \tau_\alpha)$ is covariant.  Let $\psi_\alpha: C^*(E) \rightarrow \mathbb{B}(H_\alpha)$ be the associated representation.

Since $(\pi_\alpha, \tau_\alpha)$ is covariant representation, there is an associated representation \mbox{$\psi_\alpha: C^*(E) \rightarrow \mathbb{B}(H_\alpha)$}.

If $\alpha \in E^*$ with $s(\alpha) \in E^0_{sng}$, with the obvious modifications we can build $\psi_\alpha: C^*(E) \rightarrow \mathbb{B}(H_\alpha)$ as before; in this case, $H_\alpha^{-n} = 0$ for $n > |\alpha|$.  For the details see \cite[Section 12]{KatsuraTGA3}.  Let $E^{\leq \infty} = E^\infty \cup \{ \alpha \in E^* : s(\alpha) \in E^0_{sng} \}$. Define
\[ H^n = \bigoplus_{\alpha \in E^{^{\leq \infty}}} H_\alpha^n, \qquad \text{and} \qquad H = \bigoplus_{\alpha \in E^{^{\leq \infty}}} H_\alpha = \bigoplus_{n \in \mathbb{Z}} H^n. \]
Then we have a representation of $C^*(E)$ on $\mathbb{B}(H)$ given by
\[ \psi = \bigoplus_{\alpha \in E^{^{\leq \infty}}} \psi_\alpha : C^*(E) \rightarrow \mathbb{B}(H). \]

We claim $\psi$ is injective.  Note that the map \, $\bigoplus_\alpha \pi_\alpha : C_0(E^0) \rightarrow C^*(E)$ is injective since for every $v \in E^0$, there is an $\alpha \in E^{\leq \infty}$ such that $r(\alpha) = v$.  Moreover, $\psi_\alpha$ is gauge invariant for each $\alpha \in E^\infty$ since for $a \in C_0(E^0)$, $\pi_\alpha(a)(H_\alpha^n) \subseteq H_\alpha^n$ and for $\xi \in X(E)$, $\tau_\alpha(\xi)(H_\alpha^n) \subseteq H_\alpha^{n+1}$.  Hence $\psi$ is injective by the Gauge Invariant Uniqueness Theorem (Theorem \ref{thm:GIUT}).

\section{Finiteness}\label{sec:Finiteness}

\begin{theorem}\label{thm:Finiteness}
Suppose $E = (E^0, E^1, r, s)$ is a compact topological graph with no sinks; i.e.\ $E^0$ and $E^1$ are compact and $s$ is surjective.  If $C^*(E)$ is finite, then $s$ is a homeomorphism and $r$ is surjective.
\end{theorem}

\begin{proof}
Since $E^1$ is compact, $1 \in C(E^1)$.  With the notation from the previous section, define $T_\alpha = \tau_{\alpha}(1) \in \mathbb{B}(\ell^2(H_\alpha))$ for $\alpha \in E^{\leq \infty}$ and define $T = \bigoplus_{\alpha} T_\alpha \in \mathbb{B}(H)$.  Recall that $\psi = \bigoplus_{\alpha} \psi_\alpha$ is a faithful representation of $C^*(E)$.  Identifying $C^*(E)$ with the range of $\psi$, we have $C^*(E) \subseteq \mathbb{B}(H)$ and $T \in C^*(E)$.

Fix $\alpha \in E^{\leq \infty}$.  Then $\Gamma_\alpha$ can be viewed as a directed tree $(F^0, F^1, r, s)$ by setting $F^0 = F^1 = \Gamma_\alpha$, $r = \mathrm{id}$, and $s = \sigma$.  Then $T_\alpha$ is a weighted shift on $\Gamma_\alpha$ with constant weights 1.  For each $\beta \in \Gamma_\alpha$,
\[ \|T_\alpha \delta_\beta\| = \left\|\sum_{e \in s^{-1}(r(\beta))} \delta_{e\beta}\right\| = |s^{-1}(r(\beta))|^{1/2} \geq 1 \]
since $E$ has no sinks.  Therefore, by Corollary \ref{cor:BoundedBelow}, $T_\alpha$ is bounded below by 1 and hence there is an $S_\alpha \in \mathbb{B}(\ell^2(\Gamma_\alpha))$ with $S_\alpha T_\alpha = 1$ and $\| S_\alpha \| \leq 1$.

Set $S = \bigoplus_\alpha S_\alpha \in \mathbb{B}(H)$.  Then $S T = 1$ and $T$ is left-invertible.  Since $C^*(E)$ is finite, we also have $T$ is right invertible.  Therefore, $T_\alpha S_\alpha = 1$ and $T_\alpha$ is surjective for every $\alpha \in E^{\leq \infty}$.  If $v \in E^0$ is a source, then $v \in \Gamma_v$ is also a source.  But this contradicts Corollary \ref{cor:SurjectiveShift} since $T_\alpha$ is surjective. Therefore, $r$ is surjective.

Similarly, if $e, f \in E^1$ with $s(e) = s(f)$, then choose $\alpha \in E^{\leq \infty}$ with $r(\alpha) = s(e)$.  Now, $e\alpha, f\alpha \in \Gamma_\alpha$ and $\sigma(e\alpha) = \sigma(f\alpha)$.  Since $T_\alpha$ is surjective, $\sigma|_{\Gamma_\alpha}$ is injective by Corollary \ref{cor:SurjectiveShift}.  Hence $e\alpha = f\alpha$ and $e = f$.  Thus $s$ is injective.
\end{proof}

\begin{remark}
The converse of Theorem \ref{thm:Finiteness} fails.  Take $E^0 = E^1 = \mathbb{Z} \cup \{\pm \infty\}$ with the usual topology and define $r(n) = n$ and $s(n) = n + 1$ for each $n \in \mathbb{Z} \cup \{\pm \infty\}$.  Then $r$ and $s$ are both homeomorphisms but $C^*(E)$ is infinite as can be seen from Theorem \ref{thm:Pimsner} or Theorem \ref{thm:TopGraphsAFEmbedding} below.
\end{remark}

We will show that when $C^*(E)$ is finite, $C^*(E) \cong C(E^\infty) \rtimes_{\sigma} \mathbb{Z}$.  First we establish a simple lemma.

\begin{lemma}\label{lem:ProjectiveLimit}
Suppose $E$ is a compact topological graph with no sinks.  If $s$ is injective and $r$ is surjective, then the maps $\rho_n: E^\infty \rightarrow E^0$ given by $\rho_n(\alpha) = r(\alpha(n))$ induce a homeomorphism $\rho: E^\infty \rightarrow \underset{\longleftarrow}{\lim}\,( E^0, r \circ s^{-1})$.
\end{lemma}

\begin{proof}
It is clear that $r \circ s^{-1} \circ \rho_{n+1} = \rho_n$ for every $n$ and hence $\rho$ is well-defined.  Since $\rho_n$ is surjective for every $n$, $\rho$ is also surjective.  Since $E^\infty$ is compact and Hausdorff, it suffices to show $\rho$ is injective. But if $\rho(\alpha) = \rho(\beta)$ for every $\alpha = \beta$, then $\rho_n(\alpha) = \rho_n(\beta)$ for every $n$; i.e.\ $r(\alpha(n)) = r(\beta(n))$ for every $n$.  Now,
\[ \alpha(n) = s^{-1}(r(\alpha(n+1))) = s^{-1}(r(\beta(n+1))) = \beta(n) \]
for every $n$ and $\alpha = \beta$.
\end{proof}

\begin{theorem}\label{thm:InfinitePathCrossedProduct}
Suppose $E$ is a compact topological graph with no sinks.  If $s$ is injective and $r$ is surjective (e.g.\ if\, $C^*(E)$ is finite), then $\sigma: E^\infty \rightarrow E^\infty$ is a homeomorphism and there is a natural isomorphism $C^*(E) \rightarrow C(E^\infty) \rtimes_\sigma \mathbb{Z}$.
\end{theorem}

\begin{proof}
Since $r$ is injective and $s$ is surjective, $\sigma$ is a homeomorphism.  Let $u$ be a unitary in $C(E^\infty) \rtimes_\sigma \mathbb{Z}$ such that $uau^* = a \circ \sigma$ for each $a \in C(E^\infty)$.  Define $\pi:C(E^0) \rightarrow C(E^\infty) \rtimes_\sigma \mathbb{Z}$ and $\tau:C(E^1) \rightarrow C(E^\infty) \rtimes_\sigma \mathbb{Z}$ by $\pi(a) = a \circ r$ and $\tau(\xi) = (\xi \circ \operatorname{ev}_1) u$ for $a \in C(E^0)$ and $\xi \in C(E^1)$, where $\operatorname{ev}_1:E^\infty \rightarrow E^1$ is the evaluation map $\alpha \mapsto \alpha(1)$.

A tedious but straight forward computation show $(\pi, \tau)$ is a covariant representation and hence induces a morphism $\psi: C^*(E) \rightarrow C(E^\infty) \rtimes_\sigma \mathbb{Z}$.  By the Gauge Invariance Uniqueness Theorem, $\psi$ is injective.  It remains to show surjectivity.  Note that $\tau(1) = u$.  Hence, with the notation of Lemma \ref{lem:ProjectiveLimit}, it suffices to show $a \circ \rho_n$ is in the range of $\psi$ for every $a \in C(E^0)$ and $n \geq 1$.  But
\[ a \circ \rho_n = a \circ r \circ \sigma^{n-1} = u^{n-1} (a \circ r) u^{1-n} = \tau(1)^{n-1} \pi(a) \tau(1)^{1-n}. \]
Hence $\psi$ is surjective.
\end{proof}

Using Theorem \ref{thm:InfinitePathCrossedProduct} together with Theorem \ref{thm:Pimsner}), we will be able to characterize the AF-embeddability of $C^*(E)$ for compact graphs $E$ with no sinks.  First we need a simple lemma regarding pseduoperiodic points (see condition (5) of Theorem \ref{thm:Pimsner} for the definition).

\begin{lemma}\label{lem:PseudoperiodicPoints}
Suppose $X$ is a compact metric space and $f: X \rightarrow X$ is continuous and surjective.  Define $\widehat{X} = \underset{\longleftarrow}{\lim}\, (X, f)$ and let $\widehat{f}$ denote the homeomorphism of $\widehat{X}$ induced by $f$.  Then every point of $(X, f)$ is pseudo\-periodic if and only if every point of $(\widehat{X}, \widehat{f})$ is pseudoperiodic.
\end{lemma}

\begin{proof}
Write
\[ \widehat{X} = \left\{(x_n)_{n=1}^\infty \in \prod_{n=1}^\infty X : x_n = f(x_{n+1}) \text{ for every } n \geq 1 \right\}. \]
Then $\widehat{f}: \widehat{X} \rightarrow \widehat{X}$ is given by $\widehat{f}((x_n)_n) = (f(x_n))_n$.  Let $d$ be any fixed metric on $X$ with $d(x, y) \leq 1$ for every $x, y \in X$.  Then $\widehat{X}$ is a metric space with metric
\[ \widehat{d}((x_n), (y_n)) = \sum_{n=1}^\infty 2^{-n} d(x_n, y_n). \]

Suppose first every point of $(\widehat{X}, \widehat{f})$ is pseduoperiodic.  Fix $x_1 \in X$ and $\varepsilon > 0$.  Since $f$ is surjective, there is an $\widehat{x}_1 \in \widehat{X}$ with $\widehat{x}_1(1) = x_1$.  Choose $\widehat{x}_2, \widehat{x}_3, \ldots, \widehat{x}_n$ such that $\widehat{d}(\widehat{f}(\widehat{x}_i), \widehat{x}_{i+1}) < \varepsilon$ for $i = 1, \ldots, n$.  Define $x_i = \widehat{x}_i(1)$ for $i = 2, \ldots, n$.  Then
\[ d(f(x_i), x_{i+1}) = d(\widehat{f}(\widehat{x}_i)(1), \widehat{x}_{i+1}(1)) \leq 2 \widehat{d}(\widehat{f}(\widehat{x}_i), \widehat{x}_{i+1}) < 2 \varepsilon. \]
Hence $x_1$ is pseudoperiodic.

Now suppose every point of $(X, f)$ is pseudoperiodic.  Fix $\widehat{x}_1 \in \widehat{X}$ and $\varepsilon > 0$.  By the uniform continuity of $f$, we may find a sequence \mbox{$\delta_1 = \varepsilon > \delta_2 > \delta_3 > \cdots > 0$} such that $d(f(x), f(y)) < \delta_n$ whenever $x, y \in X$ with $d(x, y) < \delta_{n+1}$.

Choose $N \geq 1$ such that $2^{-N} < \varepsilon$.  Since $y_1 := \widehat{x}_1(N)$ is pseduoperiodic point for $(X, f)$, there are points $y_2, \ldots y_n \in \widehat{X}$ such that $d(y_i, y_{i+1}) < \delta_N$ for each $i = 1, \ldots, n$.  Since $f$ is surjective, there are points $\widehat{x}_2, \widehat{x}_3, \ldots, \widehat{x}_n \in \widehat{X}$ such that $\widehat{x}_i(N) = y_i$.  Now it is to easy see $d(\widehat{x}_i, \widehat{x}_{i+1}) < 2\varepsilon$ for each $i = 1, \ldots, n$ and hence $\widehat{x}_1$ is pseudoperiodic.
\end{proof}

The next definition will give our ``combinatorial'' characterization of finiteness for topological graph algebras.  It can be thought of as a graph theoretic version of pseduoperiodic points.

\begin{definition}
Suppose $E$ is a topological graph and $\varepsilon > 0$.  Let $d$ be a metric on $E^0$ compatible with the topology.  An \emph{$\varepsilon$-pseudopath} in $E$ is a finite sequence $\alpha = (e_n, \ldots, e_1)$ in $E^1$ such that for each $i = 1, \ldots, n - 1$, $d(r(e_i), s(e_{i+1})) < \varepsilon$.  We write $s(\alpha) = s(e_1)$ and $r(\alpha) = r(e_n)$.  An $\varepsilon$-pseduopath $\alpha$ is called an \emph{$\varepsilon$-pseduoloop based at $s(\alpha)$} if $d(r(\alpha), s(\alpha)) < \varepsilon$.
\end{definition}

\begin{theorem}\label{thm:TopGraphsAFEmbedding}
If $E$ is a compact graph with no sinks, then the following are equivalent:
\begin{enumerate}
  \item $C^*(E)$ is AF-embeddable;
  \item $C^*(E)$ is quasidiagonal;
  \item $C^*(E)$ is stably finite;
  \item $C^*(E)$ is finite;
  \item $s$ is injective and for every $v \in E^0$ and $\varepsilon > 0$, there is an $\varepsilon$-pseudoloop in $E$ based at $v$.
\end{enumerate}
\end{theorem}

\begin{proof}
The implications $(1) \Rightarrow (2) \Rightarrow (3) \Rightarrow (4)$ are true for arbitrary \mbox{$C^*$-algebras} \cite[Propositions 7.1.9, 7.1.10, and 7.1.15]{BrownOzawa}).  If $C^*(E)$ is finite, then by Theorem \ref{thm:InfinitePathCrossedProduct} $C^*(E) \cong C(E^\infty) \rtimes_\sigma \mathbb{Z}$.  By Theorem \ref{thm:Pimsner}, every point in $E^\infty$ is pseudoperiodic for $\sigma$ and hence every point of $E^0$ is pseduoperiodic for $r \circ s^{-1}$ by Lemmas \ref{lem:ProjectiveLimit} and \ref{lem:PseudoperiodicPoints}. Fix a metric $d$ on $E^0$.  Given $v \in E^0$ and $\varepsilon > 0$, choose $v_1 = v, v_2, \ldots, v_n \in E^0$ such that $d(r(s^{-1}(v_i)), v_{i+1}) < \varepsilon$ for every $i = 1, \ldots, n$.  Set $e_i = s^{-1}(v_i)$.  Then $(e_n, \ldots, e_2, e_1)$ is a pseduoloop in $E$ based at $v$.

Now suppose (5) holds.  Condition (5) implies $r$ has dense range and since $E^1$ is compact, $r$ is surjective.  Now, $C^*(E) \cong C(E^\infty) \rtimes_\sigma \mathbb{Z}$ by Theorem \ref{thm:InfinitePathCrossedProduct}.  Given $v \in E^0$ and $\varepsilon > 0$, choose an $\varepsilon$-pseduoloop $(e_n, \ldots, e_1)$ based at $v$.  Set $v_i = s(e_i)$ and note that $d(r(s^{-1}(v_i)), v_{i+1}) = d(r(e_i), s(e_{i+1})) < \varepsilon$. Hence by Theorem \ref{thm:Pimsner} and Lemma \ref{lem:PseudoperiodicPoints}, $C^*(E)$ is AF-embeddable.
\end{proof}

\begin{remark}
In Theorem \ref{thm:TopGraphsAFEmbedding}, the assumptions that $E$ is compact and has no sinks are necessary even when $E$ is discrete.  Consider the graphs below
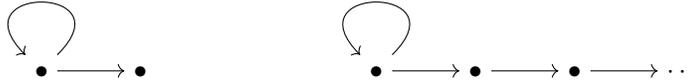
\[ \begin{tikzcd} \bullet \arrow{r} \arrow[loop]{} & \bullet & & & \bullet \arrow{r} \arrow[loop]{} & \bullet \arrow{r} & \bullet \arrow{r} & \cdots \end{tikzcd} \]
In both cases, the $C^*$-algebra defined by the graph is finite but $s$ is not injective.
\end{remark}

\end{document}